\documentclass[12pt,a4paper,leqno]{amsart}

\usepackage[cp1250]{inputenc}

\usepackage{amsmath}
\usepackage{amsthm}
\usepackage{amssymb}

\newcommand{\RO}{\operatorname{RO}}
\newcommand{\ro}{\mathsf{ro}}

\newcommand{\cl}{\operatorname{cl}}

\renewcommand{\epsilon}{\varepsilon}
\renewcommand{\int}{\operatorname{int}}
\renewcommand{\phi}{\varphi}

\newtheorem{thm}{Theorem}
\newtheorem{pro}[thm]{Proposition}
\newtheorem{lem}[thm]{Lemma}
\newtheorem{cor}[thm]{Corollary}
\newtheorem*{fac}{Fact}

\title{\textsc{The Niemytzki plane is $\varkappa$-metrizable}}


\author{Wojciech Bielas}
\address{Institute of Mathematics, University of Silesia, ul. Bankowa 14, 40-007 Katowice}
\email{wojciech.bielas@us.edu.pl}
\email{andrzej.kucharski@us.edu.pl}
\email{szymon.plewik@us.edu.pl}
\author{Andrzej Kucharski}

\author{Szymon Plewik}

\begin{document}

\maketitle

\begin{abstract} We try to  explain the differences between the concepts of stratifiable space and $\varkappa$-metrizable space. In particular, we give a characterization of $\varkappa$-metrizable spaces  which is modelled on  Chigogidze's  characterization.  Moreover, we present a $\varkappa$-metric for  the Niemytzki plane, using the properties of the Euclidean metric.   \end{abstract} 

\section{Introduction} 
The purpose of this paper is to present elementary or alternative proofs of some facts about  $\varkappa$-metrizable spaces. Our approach is focused on completely regular spaces, as it was intended by Shchepin, compare \cite[p. 164]{s76}. The class of $\varkappa$-metrizable spaces was introduced by   Shchepin  as an axiomatic theory based on four axioms which have been denoted by
(K1)--(K4). This class is a generalization of metric spaces and it is wide enough to contain many important classes of spaces that are not metrizable, see \cite{s76} and \cite{shc}. We analyse the relationships between axioms of a $\varkappa$-metric. To emphasize our  motivations, let us quote Sierpi\'nski's book \cite{sie}. 
\begin{quotation}
	\textit{The theorems of any geometry (e.g. Euclidean) follow, as is well known, from a number of axioms, i.e. hypotheses about the space considered, and from accepted definitions. A given theorem may be a consequence of some of the axioms and may not require all of them.}
\end{quotation}
 As a by-product, we get a  class of spaces which  we call $\ro$-stratifiable.
 We could not find  a publication in which $\ro$-stratifiable spaces are examined. We deal with the Niemytzki plane. As will be shown, this case  indicates that certain  properties of the Euclidean metric are crucial in a more general  setting.

 Our notations are standard, following \cite{en} or \cite{ss}. In spite of that, an open subset $U$ of a topological space $ X$ is called \textit{{regular open}} whenever it is the interior of a closed set: in  other words $U$ is a regular open set whenever $U=\int_X\cl_X(U); \text{\; or in brief \; }  U=\int\cl U.$  We   denote the family of all regular open subsets of $X$ by $\RO(X)$.  The complement of a regular open set is called  \textit{regular closed}.  So,  $F\subseteq X$ is a regular closed set whenever $F=\cl\int F.$ A subset $G$ of a topological spaces $X$ is a \textit{co-zero} set whenever there exists a continuous function $f\colon X \to [0,1]$ such that $G= f^{-1}((0,1])$. 

Let  $X$ be a topological space. Fix  a family  $\mathcal{B}$ which consists of open sets  and  a  family of  functions  $\{f_U\colon U\in\mathcal{B} \}$, where $f_U: X\to   [0,1]$. Consider the following conditions. 
\begin{itemize}
\item[(1)] If  $U\in\mathcal{B}$, then $U=f_U^{-1}((0,1]).$
\item[(2)]If $U,V\in\mathcal{B}$ and $U\subseteq V$, then $f_U(x)\leqslant f_V(x)$ for any $x\in X$.
\item[(3)] For any $U\in\mathcal{B}$, the function $f_U: X \to [0,1]$ is  continuous. \end{itemize}

According to  \cite[Theorem 5.2]{bor} a $T_1$-space $X$ is \textit{stratifiable} whenever it fulfils  conditions (1)--(3)  for $\mathcal{B}$ consisting of all open sets. 
Let us add that   Ceder's $M_3$-spaces were given a new name  \textit{stratifiable spaces} in the paper \cite{bor}, compare \cite{ced}. 
Following  Shchepin (see \cite[p. 164]{s76}, compare \cite[p. 407]{shc})  a completely regular space $X$ is called \textit{$\varkappa$-metrizable} whenever there exists a family of continuous  functions $\{f_U\colon U \in \RO (X)\}$ which satisfies conditions (1)--(3) and the following condition.
\begin{itemize}
\item[(4)] Let $(U^\alpha)$  be a decreasing sequence of {regular open} sets. If  $$W= {\int \bigcap_\alpha U^\alpha},$$ then  $ f_W(x) =\inf_\alpha f_{U^\alpha} (x), $ for any $x\in X.$ \end{itemize}   
Introducing the concept of a $\varkappa$-metrizable space,  Shchepin utilized regular closed sets. His axioms (K1)--(K4) for $\varkappa$-metric, see \cite[p. 164]{shc}, are direct translations, via de Morgan's laws,  of our conditions (1)--(4). Nonetheless, a family of continuous  functions $\{f_U\colon U \in \RO (X)\}$ which satisfies conditions (1)--(4) we call  a $\varkappa$-\textit{metric}. But a $T_0$-space with a $\varkappa$-metric and such that the family $\RO(X)$ constitutes a base  is called $\varkappa$-\textit{metrizable}. This definition is equivalent to Shchepin's one. To see this, let   
  $\mathcal{B}$ be a base for  a $T_0$-space $X$  and let a family $\{f_U\colon U\in \mathcal{B} \}$    fulfil  conditions $(1)$ and $(3)$. 
	Suppose $x$ and $y$ are different points of $X$. If  {$U$ in $\mathcal{B}$} is such that  $x\in U$ and $y\notin U$, then $$ \textstyle x\in f_U^{-1} ((\frac{f_U(x)}{2}, 1]) \mbox{ and }  y\in f_U^{-1} ([0, \frac{f_U(x)}{2})), $$ which shows that  $X$ is a Hausdorff space because {of continuity of $f_U$}. Thus,   $X$ is  completely regular, being a $T_1$-space with the  base $\{U \colon U\in \RO(X) \}$ which consists of co-zero sets.

	The paper is organized as follows. Above, we have provided conditions equivalent to the definition of a $\varkappa$-metrizable space.
	In the second part we introduce the concepts: $\mathcal{B}$-stratification, $\ro$-stratifiable and $\mathcal{B}$-approximation. Propositions \ref{p1}  and \ref{pr3} thoroughly explain the relationships between these concepts. We show the reason why the Sorgenfrey line is 
	not stratifiable, being  $\varkappa$-metrizable. Also, we show that the double arrow space is $\ro$-stratifiable, but not $\varkappa$-metrizable,  see ends of parts 2 and 3.  This indicates that condition $(4)$ is independent of conditions (1)--(3). 
Propositions  \ref{l1} and \ref{l2}  establish a characterization of a $\varkappa$-metrizable space. In the last part we discuss the properties of the Niemytzki plane.

 \section{$\mathcal{B}$-approximations and $\ro$-stratifiable spaces }
 If $X$ is a $T_0$-space and  a  family of  functions $\{f_U\colon  U\in\mathcal{B}\}$   fulfils  conditions (1)--(3), where  $f_U\colon X \to [0,1]$ for all $U$ in $\mathcal{B}$, then  we will call this family a \textit{$\mathcal{B}$-stratification}. If there exists a $\mathcal{B}$-stratification, then  the space $X$ is said to be \textit{$\mathcal{B}$-stratifiable}. Clearly, if $\mathcal{A} \subseteq \mathcal{B}$ and a space  is $\mathcal{B}$-stratifiable, then it is also $\mathcal{A}$-stratifiable.  If $\mathcal{B} = \RO(X)$, then we will  say  that $ X $  is  \textit{$\ro$-stratifiable}  instead of   
$ \RO(X)$-stratifiable. 
If a space $X$ is $\ro$-stratifiable,  then  any regular open set of $X$ is  a co-zero set by conditions (1) and $(3)$. Moreover, if conditions (1) and $(3)$ are fulfilled for $\mathcal{B}=\RO(X)$, then the space $X$ is $\varkappa$-normal: recall that a completely regular space is  \textit{$\varkappa$-normal} whenever any pair of non-empty disjoint  and regular closed sets can be separated by disjoint open sets, see \cite{shc}, compare \cite{bp}. Thus, if $F$ and $G$ are disjoint regular closed sets, then $F=f^{-1}(0)$ and $G=g^{-1}(0)$, where functions $f, g \colon X \to [0,1]$ are continuous. Then preimages of $[0, \frac12)$ and $(\frac12,1]$ via the continuous function 
 $\frac{f}{f+g}$ separate $F$ and $G$.  Under the additional assumption that each regular closed subset of $X$ is a $G_\delta$ set the reverse is  true,  which can be checked by modifying a proof of   Urysohn's lemma. This additional assumption is necessary as shown below.

There are compact Hausdorff  spaces which are not $\ro$-stratifiable. For instance, a compact Hausdorff space, containing a regular open subset which is not a co-zero set. Obviously, any such space   is $\varkappa$-normal, being a normal space. To see an example, let us consider  $$\textstyle   Y=\{\alpha\colon \alpha \leqslant \omega_1\} \cup \{\frac1n\colon n >0\}$$ and  the linear order $(Y,<)$ such that it is  the restriction of the well order of the ordinals on $\{\alpha\colon \alpha \leqslant \omega_1\}$ and  it  inherits the order  from the real line on $\{\frac1n\colon n >0\}$ and if $\alpha \leqslant \omega_1$ and $n>0$, then $\alpha < \frac1n$. The linear topology on $Y$ which is generated by $<$ is  compact and Hausdorff. In this topology,  there are regular open sets which are not co-zero sets, for example  $\{\alpha\colon \alpha < \omega_1 \} $.

In fact, the above reasoning does not use condition (2).  Let us add that there are many results about $\varkappa$-normal spaces, for example compare \cite{kal}. Also, there exist many examples of a completely regular space which is not $\varkappa$-normal, e.g., the ones which can be built using  a technique called the
Jones' machine, compare \cite{kp} or \cite{bp}. 

It was noted in \cite[pp. 106--107]{ced} that the Sorgenfrey line $\mathbb S$, i.e., the real line with a topology generated by intervals $[a,b)$,  is not stratifiable, being a paracompact and perfectly normal space: in other words,  if $\mathcal{B}$ is the family of all open sets in $\mathbb S$, then no  family of functions fulfils conditions (1)--(3) with respect to $\mathcal{B}$. 
Nonetheless, the family  consisting of characteristic functions of closed-open sets of $\mathbb S$  fulfils  conditions (1)--(3).  Proposition \ref{pr1} will make clear these  facts, using  the union of the family of all left closed interval  and the family of all open intervals with rational end-points. 
\begin{pro}\label{pr1} If  $ \textstyle \mathcal A= \{[x,y): x<y\} \cup\{(x,y): x,y \in \mathbb Q\}, $ then the Sorgenfrey line  is not $ \mathcal A$-stratifiable.
\end{pro}
\begin{proof} Suppose that a family $\{f_U: U \in  \mathcal A\}$ is an $\mathcal{A}$-stratification, i.e., it  witnesses that $\mathbb S$ is $\mathcal{A}$-stratifiable.  If $(a,a+2)\in \mathcal{A}$ and  $n>0$, then  put $$ \textstyle R_{n}= (a,a+2) \cap \{x\colon f_{[x, x+1)}(x) > \frac1n \}. $$ Since $(a,a+1) \subseteq \bigcup \{R_{n}\colon n>0\}$,  using  the Baire category theorem,  choose  $n$ such that $$(a,a+1) \cap \int \cl R_{n} \not= \emptyset,$$ where the interior and the closure are taken with respect to the Euclidean topology. Next, choose a decreasing sequence $(x_k)$ converging to the  rational number $x$ such that  $x_k \in(a,a+1) \cap R_{n}$. 
Thus each number $1+x_k \in (x,a+2)$, so by   condition $(2)$, we obtain  
$$\textstyle f_{(x,a+2)}(x_k)  \geqslant f_{[x_k,x_k+1)}(x_k)> \frac1n.  $$ 
Since $x\notin (x,a+2)$, by   condition $(1)$, we obtain  $f_{(x,a+2)}(x)=0$, which  contradicts    the continuity of   $f_{(x,a+2)}$.  
\end{proof}

It is known that the Sorgenfrey line $\mathbb S$ is a $\varkappa$-metrizable space, compare \cite[p. 507]{stt}.  Therefore  the space   $\mathbb S$ is $\ro$-stratifiable. We present an alternative proof, using the sequential criterion for the continuity of a function.   If $U$ is a  {regular open} subset of $\mathbb S$, then put $$f_U(x) = \begin{cases} \sup\{q-x\colon [x,q) \subseteq U \cap [x,x+1)\},&\text{ when } x \in U;\\ 0, & \text{ when } x \notin U.     \end{cases}$$   By 
the definition, the family $\{f_U\colon U \in \RO(\mathbb S)\}$ fulfils  conditions  $(1)$ and $(2)$. To verify   condition $(3)$, we shall check that each function $f_U\colon\mathbb S \to [0,1]$ is continuous. Indeed, suppose that a sequence $(x_n)$ is convergent to $x$. Since we consider convergence in $\mathbb S$, we can assume that always $x\leqslant x_n$. Thus, if $x\in U$, then by the definition of $f_U$, the sequence $(f_U(x_n))$ converges to $f_U(x)$. But if  $x\notin U$, then take a decreasing sequence $(y_n)$ converging to $x$ such that  $y_n \notin U$ and $x_n <y_n$ which is possible since $U$ is {regular open} in $\mathbb S$.  Then, again using the definition of $f_U$ we check that $f_U(x_n) \leqslant y_n -x$ which implies that the sequence $(f_U(x_n))$ converges to $0=f_U(x)$.

Now, we will slightly modify the definition of a stratifiable  space which  was proposed in \cite[p. 1]{bor}. Let $\mathbb I = (0,1)\cap\mathbb Q$ be the set of all rational numbers from the open unit interval.  Let us assume that a family $\{U_q\colon q\in \mathbb I \} $, consisting of open sets, is assigned to a  set $U\in\mathcal{B}$.
A collection of such families  $\{U_q\colon q\in \mathbb I \} $ will be called a $\mathcal{B}$-\textit{approximation} if the following three conditions  are fulfilled.

\begin{itemize}
\item[$(a)$] If $U\in\mathcal{B}$, then $U=\bigcup \{U_q\colon q\in \mathbb I\}$.
\item[$(b)$] If $U,V\in\mathcal{B}$  and $U\subseteq V$, then $U_q\subseteq V_q$.
\item[$(c)$] If $U\in\mathcal{B}$ and   $p<q$, then $ \cl (U_q)  \subseteq U_p$.
\end{itemize}

Any $\RO(X)$-approximation about which it is assumed that the indexed sets are  open, can be improved to an  $\RO(X)$-approximation with all indexed sets being regular open, for example by  the substitution $ U_q \mapsto \int \cl (U_q) $. 
The following propositions establish a connection between $\mathcal{B}$-approximations and $\mathcal{B}$-stratifications.

\begin{pro}\label{p1} If  a family $\{f_U\colon U\in\mathcal{B} \}$ is a $\mathcal{B}$-stratification, then the assignment  \, $U_q= f_U^{-1}((q,1])$, where $U \in \mathcal{B}$ and $q\in \mathbb I$,  establishes a $\mathcal{B}$-approximation.  \end{pro}
\begin{proof} Sets $U_q$ are open, since each  $f_U$ is a continuous function.   By the definition of $f_U$,  conditions $(1)$ and  $(a)$ are equivalent. For the same reasons, conditions $(2)$ and  $(b)$ are equivalent. If $p<q $, then we have $$ \cl (U_q) \subseteq f^{-1}_U([q,1]) \subseteq U_p,$$ since $f_U$ is a continuous function.\end{proof}

\begin{pro}\label{pr3} If  a collection of families  $\{U_q\colon q\in \mathbb I \} $ is  a $\mathcal{B}$-approximation, then the family $\{f_U\colon U\in\mathcal{B} \}$, where $$f_U(x) = \begin{cases} \sup\{q \in \mathbb I\colon x \in U_q\},& \text{ when } x \in U;\\ 0, & \text{ when } x \notin U,     \end{cases}$$  is a $\mathcal{B}$-stratification.   \end{pro}

\begin{proof} Clearly,   $(b)$ implies $(2)$. 
Each function $f_U$ is upper semi-continuous, since  $$ f_U^{-1} ([0,q)) = \bigcup\{X \setminus \cl (U_p)\colon p<q\}.$$ Indeed, if $f_U(x) <q$, then take $p_1, p_2 \in \mathbb I$ such that $f_U(x)<p_1<p_2 <q$. Since  $(c)$, we have $x\notin U_{p_1} \supseteq \cl( U_{p_2})$. But when $p<q$ and $x \notin \cl( U_p)$, we have $x\notin U_p$. Again by 
$(c)$ and the definition of $f_U$,  check that  $f_U(x) \leqslant p$.  
Each function $f_U$ is lower  semi-continuous, since  $$ f_U^{-1} ((q,1]) = \bigcup\{U_p\colon p>q\}.$$ Indeed, if $f_U(x) >q$, then, by the definition of $f_U$,  there exists  $p>q$ such that $x \in U_p.$  But when $ x \in U_p$ and $p>q$, then  $f_U(x) \geqslant p >q.$ We have showed that any function $f_U$ is continuous, whenever $U\in \mathcal{B}$.  
Obviously, $U=\bigcup \{U_q\colon q\in \mathbb I\}$ implies $U=f_U^{-1}((0,1]).$    
 \end{proof} 

 Borges \cite[Theorem 5.2]{bor} characterized a stratifiable space as a space with a $\mathcal{B}$-stratification, where $\mathcal{B}$ consists of all non-empty open sets. So, propositions \ref{p1} and \ref{pr3} establish another characterization of stratifiable spaces. 

Consider the lexicographic order on $\mathbb D = [0,1] \times \{0,1\}$. Note that  $\mathbb D$ with the order topology is well known as the double arrow space or the two arrows space. Observe that regular open subsets of $\mathbb D$ are union of pairwise disjoint closed-open intervals. Indeed, if $U \subseteq \mathbb D$ is an open set, then let $U_x$ be the union of all open interval $I \subseteq U$ such that $x\in I$. If $U$ is regular open, then sets $U_x$ are closed open and constitute a partition of $U$. Since $U_x= ((a,0), (b,1))= [(a,1), (b,0)]$ for each $x\in U$,  put $$f_U(x)=\begin{cases} 0,& \mbox{ when } x\notin U;\\
b-a, & \mbox{ when } x \in U_x = [(a,1), (b,0)];\\ 1, & \mbox{ when } x\in U\cap \{(0,0), (1,1)\}.\end{cases}$$  Then,  check that the family $\{f_U\colon U \in \RO(\mathbb D) \}$ witnesses that the double arrow space is $\ro$-stratifiable. 

\section{On $\varkappa$-metrizable spaces}

The notion of a $\varkappa$-metrizable space (a $\varkappa$-metric space) was introduced by  Shchepin, see \cite{s76}, compare \cite{shc}. In   \cite{chi},  Chigogidze gave a characterization of $\varkappa$-metrizable spaces, but as it was noted in \cite{wic}: \textit{This article is an announcement of results; proofs are not included}. So,  we propose a slight modification of the characterization from the paper \cite{chi}. Assume that a space $X$ is  completely regular and $\ro$-stratifiable. Fix an $\RO(X)$-stratification $\{f_U\colon U \in \RO(X) \}$. Let   $\{\{U_q\colon q\in \mathbb I \}\colon U \in \RO(X)\}$ be the  $\RO(X)$-approximation assigned via the rule $U_q = f_U^{-1}((q,1])$. Then consider the following conditions, where a sequence  $(U^\alpha)$ may be transfinite. 
\begin{itemize}
	\item[$(d)$] If $(U^\alpha)$ is a decreasing sequence of {regular open} sets and $p<q$, then $$  \bigcap_\alpha\cl( U^\alpha_q)\subseteq(\int \bigcap_\alpha U^\alpha)_p.
  $$
 \end{itemize}

 Because of \cite[Theorem 18]{shc} the double arrow space $\mathbb D$, being compact with countable character and with the weight continuum, is not $\varkappa$-metrizable. Thus, the class of all $\ro$-stratifiable spaces is essentially wider than the class of all 
$\varkappa$-metrizable spaces.

\begin{pro}\label{l1}  If the $\RO(X)$-approximation 
	 $\{\{U_q\colon q\in \mathbb I \}\colon U \in \RO(X)\}$ fulfils  condition $(d)$, then the $\RO(X)$-stratification  $\{f_U\colon U\in \RO(X)\}$  fulfils   condition {\upshape(4)}.
\end{pro}

\begin{proof} Let $(U^\alpha)$ be a decreasing sequence of {regular open} sets and let $$W={\int \bigcap_\alpha U^\alpha}.$$ Suppose that there exists $x$ in $X$ such that $ f_W(x) \not=\inf_\alpha f_{U^\alpha} (x).$ By  condition $(2)$, we have  $ f_W(x) <\inf_\alpha f_{U^\alpha} (x).$ Choose rationals $p<q$ such that $$ f_W(x)<p<q <\inf_\alpha f_{U^\alpha} (x).$$ Since $f_W(x) <p$ implies $x\in X \setminus \cl(W_p)$, by condition $(d)$, we get $$x\in X \setminus W_p \subseteq \bigcup_\alpha X \setminus \cl ( U_q^{\alpha}).$$ So, there exists $\beta$ such that $x \in X \setminus \cl(U_q^{\beta})$, which implies that $f_{U^{\beta}}(x)\leqslant q;   $ a contradiction. \end{proof}

\begin{pro}\label{l2}  If the $\RO(X)$-stratification  $\{f_U\colon U\in \RO(X)\}$ fulfils  condition {\upshape (4)}, then the $\RO(X)$-approximation 
	$\{\{U_q\colon q\in \mathbb I \}\colon U \in \RO(X)\}$   fulfils   condition $(d)$.
\end{pro}
\begin{proof} Let $(U^\alpha)$ be a decreasing sequence of {regular open} sets and let $$W={\int \bigcap_\alpha U^\alpha}.$$ Fix rationals $p<q$. Suppose that there exists $x\in \bigcap_\alpha\cl(U_q^\alpha) \setminus W_p $. Hence $f_W(x) \leqslant p <q$. By condition $(4)$ there exists $\beta$ such that $f_{U^\beta}(x) <q$. But the function $f_{U^\beta}$ is continuous, there exists an open $V\ni x$ such that $V \subseteq f^{-1}_{U^\beta}([0,q))$.  Therefore $V \cap U^{\beta}_q
 \not= \emptyset.$ If  $b\in V \cap U^{\beta}_q$, then $q\leqslant f_{U^\beta}(b) < q;$ a contradiction.
 \end{proof}
 
 Suppose that a family $\{f_U\colon U\in \RO(X)\}$ witnesses that a space $X$ is  $\ro$-stratifiable. This family   fulfils condition {\upshape(4)} if and only if it yields the $\RO(X)$-approximation which fulfils condition $(d)$. 
Thus, we obtain  a characterization of $\varkappa$-metrizable spaces,  looking close to the one given in \cite{chi}. 

Now, we will show why the double arrow space  $\mathbb D$ does not satisfy condition (4), which gives an alternative proof  that this space is not $\varkappa$-metrizable, compare \cite[Theorem 18]{shc}. 
Suppose that the space $\mathbb D$ is $\varkappa$-metrizable. Let  $\{\{U_q\colon q\in \mathbb I \}\colon U \in \RO(\mathbb D)\}$ be  an $\RO(\mathbb D)$-approximation. If $U=[(a,1),(b,0)] \subseteq \mathbb D$, then let $$t(U)=\sup\{p\in (0,1)\cap\mathbb Q\colon U=U_p\}.$$  Since each $U$ is a compact subset, by condition $(a)$, numbers $t(U)$ are well defined. Put
$$\textstyle R_p=\{x\in [0,\frac 1{10}]\colon t\left(((x,0),(\frac 1{5},1))\right)>p\},$$ where $p\in(0,1)\cap\mathbb Q.$ 
Note that $[0,\frac 1{10}]\subseteq\bigcup\{R_p\colon p\in(0,1)\cap\mathbb Q\}.$ By the Baire category theorem there is $p\in (0,1)\cap\mathbb Q$ such that $\int\cl R_p\ne\emptyset.$ Choose an increasing sequence $(x_k)$ which converges to $x$ and always $x_k\in R_p$.   Thus $$\textstyle 
 \cl\left(((x_k,0),(\frac 1{5},1))_p\right)=	 \cl\left(((x_k,0),(\frac 1{5}, 1))\right) = \left[(x_k, 1), (\frac 1{5},0)\right]$$ and 
$$\textstyle
\bigcap_k \left[(x_k, 1), (\frac 1{5},0)\right]=\left[(x,0),(\frac 1{5},0)\right] \nsubseteq \left((x,0),(\frac 1{5}, 1)\right);$$ which contradicts condition $(d)$.

\section{A $\varkappa$-metric for the  Niemytzki plane}
In \cite[p. 827]{shc} it was noted that V. Zaitsev showed that the Niemytzki plane is $\varkappa$-normal. A proof of this fact one can find in \cite{bp}. 
The Niemytzki plane $L$, compare \cite[p. 39]{en}, it is the closed upper half-plane  plane $L = \mathbb{R} \times [0,\infty)$  endowed with the topology generated by open discs disjoint with the real axis $L_1=\{(x,0)\colon x \in \mathbb R\}$	and all sets of the form $\{a\} \cup D$ where $D\subseteq L$ is an open disc which is tangent to $L_1$ at the point $a \in L_1$. For our needs, we  use the following notations.
Let $B((x,y), r)$ denote the open disc with   centre $(x,y)$ and  radius $r$, but    $B^*(x, r) =  \{(x,0)\}\cup B((x,r), r)$.  Put $$ \textstyle \mathcal{B}=\{B((x,y), r)\colon   (x,y)\in L\setminus L_1  \text{ and } r \leqslant y \text{ and } 0<r\leqslant 1\}$$ and 
$\mathcal{B}^*=\{B^*(x,r)\colon   0<r\leqslant 1 \mbox{ and } x\in \mathbb R \}$. 
Thus, the family  $ \frak{B}=  \mathcal{B}^*\cup\mathcal{B} $ is a base for $L$. 
\begin{fac}\label{p7} The family $\frak{B}$ is closed with respect to increasing unions.  
\end{fac}
\begin{proof} Let  $U_n= B((x_n,y_n), r_n) \in  \mathcal{B}$ and let $(U_n )$ be an increasing   sequence. Thus the sequence of reals $(r_n)$, being bounded and increasing is convergent, i.e., $r_n \to r$.  Also, the sequence $((x_n,y_n))$ has no two subsequences  which converge to different points. Indeed,  if $(x_{k_n},y_{k_n}) \to (x,y)$ and $(x_{m_n},y_{m_n}) \to (x',y')$ and $(x,y) \not= (x',y')$, then the union $\bigcup\{U_n\colon n \geqslant 0\}$ would be a disc with radius $r$ and with two different centres, which is impossible in the Euclidean metric. Thus $(x_n,y_n) \to (x,y)$ and  $\bigcup\{U_n\colon n \geqslant 0\}=B((x,y),r)$. 
If  $U_n= B^*(x_n, r_n) \in  \mathcal{B}^*$, then we get  
$(x_{n},y_{n}) \to (x,r)$ and  $\bigcup\{U_n\colon n \geqslant 0\}=B^*(x,r)$.    \end{proof}

The above proposition is surely folklore. We include it to make   elementary methods, that we use below, more understandable. So, we think the reader will have no trouble justifying that: If $U_n= B^*(x_n, r_n) \in  \mathcal{B}^*$ and  $(U_n )$ is a decreasing   sequence, then the sequence $(x_n)$ is constant and  hence $\int_L \bigcap\{U_n\colon n \geqslant 0\}$ is empty or it belongs to $\mathcal{B}^*$. 
We are in a position to define an $\RO(L)$-stratification. If $U=B((a,b), r) \in \mathcal{B}$, then put $$ f_U(x,y)= \begin{cases} r- \sqrt{(x-a)^2 + (y-b)^2}, & \mbox{ when } (x,y) \in U;\\
0, & \mbox{ for other cases}.\\ \end{cases} $$ Thus, $f_U(x,y)$ is the distance between the point $(x,y)$ and the complement of the open disc $B((a,b), r)=U$. 

If $U  = B((a,r), r) \cup \{(a,0)\} \in \mathcal{B}^*$, then put $$ f_U(x,y)= \begin{cases} r- \sqrt{(x-a)^2 + (y-r)^2}, & \mbox{ when } (x,y) \in U \mbox{ and } r\leqslant y;\\
r,& \mbox{ when } (x,y)=(a,0); \\ 
r- \frac{r|x-a|}{\sqrt{2yr- y^2}}, & \mbox{ when } (x,y) \in U \mbox{ and } y<r;\\0, & \mbox{ for other cases}.\\ \end{cases} $$ Any function $f_U$ is continuous in   $ L \setminus L_1$, with respect to the Euclidean topology,  and hence it is continuous in $ L \setminus L_1$, with respect to the Niemytzki plane. Suppose that  $\lim_{n\to \infty} (x_n, y_n) = (a,0)$ with respect to the Niemytzki plane. Without loss of generality we can assume that $(x_n,y_n) \in B( (a,\frac1n) ,\frac1n ) $ and  $\frac2n < r$. Since always   $ |x_n-a|<\sqrt{\frac2n y_n -y_n^2}$ and  $y_n \to 0$ we get $$r \geqslant f_U(x_n,y_n)= r- \frac{r|x_n-a|}{\sqrt{2y_nr- y_n^2}}\geqslant  r- \frac{r\sqrt{\frac2n - y_n}}{\sqrt{2r- y_n}} \xrightarrow[n\to \infty]{} r.$$ Thus, we have checked that for each $U \in \frak{B}$ the function  $f_U\colon L \to [0,1]$   is  continuous.

For a given regular open set  $V\in \RO(L)$   put  $$f_V(x,y)= \sup\{f_U(x,y)\colon  U\in \frak{B} \mbox{ and } U\subseteq V\}.$$  If $V\in \frak{B}$, then  both  definitions of  $f_V$ coincide. Also, if $(x,y)\in L \setminus V$, then $f_V(x,y) =0$. 

\begin{lem} \label{l8} If $(x,y)\in V \in \RO(L)$, then there exists $U\in\frak{B} $ such that $U \subseteq V$ and  $f_V(x,y)=f_U(x,y)>0.$\end{lem}

\begin{proof}	
	 Suppose  $0<f_V(x,y)= \lim_{n\to \infty} f_{U_n}(x,y)$, where $U_n \in \frak{B}$ and $U_n \subseteq V$.
If  $U_n \in \mathcal{B}$ for infinitely many $n$, then we can assume that $B((x_n,y_n), r_n) = U_n$, where $x_n \to a$ and $y_n\to b$ and $r_n\to r>0$. We get 
$U=B((a,b),r) \subseteq V$. Indeed, fix $(c,e) \in U$. Let $\epsilon >0$ be such that $d((c,e), (a,b))=r-\epsilon$, where $d$ is  the Euclidean distance.  
Choose $n$ such that $$r_n > r - \frac{\epsilon}{2}\mbox{  and } d((a,b),(x_n,y_n))< \frac{\epsilon}{2}.$$ We have $$\textstyle d((c,e), (x_n,y_n)) \leqslant d((c,e), (a,b)) +
d((a,b), (x_n,y_n)) < r- \frac{\epsilon}{2}<r_n.$$ Therefore $(c,e) \in U_n \subseteq V$. Moreover, \begin{multline*}
f_V(x,y) = \lim_{n\to \infty} f_{U_n}(x,y)= \lim_{n\to \infty} \max\{0, r_n- \sqrt{(x-x_n)^2 + (y-y_n)^2}\}\\= \max\{0, r- \sqrt{(x-a)^2 + (y-b)^2}\} =   f_U(x,y).\end{multline*} 

If  $ U_n = B^*(a_n, r_n) \in \mathcal{B}^*$ for almost all $n$, then we can assume that  $a_n \to a$ and  $r_n\to r$ and $0<y<r_n$, since the case when  $y\geqslant r_n$ for infinitely many $n$  one can reduce to the previous reasoning. Similarly as above, check that  $ B((a,r),r) \subseteq V$.  We have $U=  B^*(a,r) \subseteq V,$ since $V\in  \RO(L)$. Therefore \begin{multline*}
f_V(x,y) = \lim_{n\to \infty} f_{U_n}(x,y)= \lim_{n\to \infty} \textstyle\max\left\{0, r_n- \frac{r_n|x-a_n|}{\sqrt{2yr_n- y^2}}\right\} \\=\textstyle \max\left\{0, r- \frac{r|x-a|}{\sqrt{2yr- y^2}}\right\}=   f_U(x,y).\end{multline*}

If $y=0$, then the family $\{U\colon (x,0) \in U \in \frak{B} \}=\{U\colon (x,0) \in U \in \mathcal{B}^* \} $ is linearly ordered by the inclusion. By Proposition \ref{p7}, the union of this family is contained in $V$. So, it gives a desired base set.   
\end{proof}

\begin{pro}\label{p9} If $V \in \RO(L)$, then the function $f_V\colon L\to [0,1]$ is continuous. \end{pro}
\begin{proof}  Assume that $\lim_{n \to \infty}(x_n,y_n) = (x,y) $ with respect to the Niemytzki plane. Suppose that $\lim_{n\to \infty} f_V(x_n,y_n) >p  > f_V(x,y)$. Choose $U_n \in \frak{B}$ such that  $(x_n,y_n) \in U_n \subseteq V$ and constantly  $ f_{U_n}(x_n,y_n) >p$. Since $$U_n = B((a_n,b_n), r_n) \mbox{  or } U_n = B^*(a_n,r_n)$$
we can assume  $a_n \to a$ and $b_n\to b$ and $r_n\to r>0$.

Let
$U = B((a,b), r) \mbox{  or } U = B^*(a,r)$,  and also  $ r_n \leqslant y_n$ for infinitely many $n$, then we get
 \begin{multline*}
 p\leqslant \lim_{n\to \infty} f_{U_n}(x_n,y_n)= \lim_{n\to \infty} \max\{0, r_n- \sqrt{(x_n-a_n)^2 + (y_n- b_n)^2}\} \\= \max\{0, r- \sqrt{(x-a)^2 + (y-b)^2}\} =   f_U(x,y)\leqslant f_V(x,y) <p;\end{multline*} a contradiction. 

But if $U=B^*(a,r) \subseteq V$  and $y>0$ and $ r_n \geqslant y_n$ for infinitely many $n$, then we get
 \begin{multline*} p\leqslant 
f_V(x,y) = \lim_{n\to \infty} f_{U_n}(x,y)= \lim_{n\to \infty} \textstyle\max\left\{0, r_n- \frac{r_n|x-a_n|}{\sqrt{2yr_n- y^2}}\right\} \\=\textstyle \max\left\{0, r- \frac{r|x-a|}{\sqrt{2yr- y^2}}\right\}=   f_U(x,y) <p;\end{multline*} again we have a contradiction.

If $y=0$, then $a_n \to x$ and $U= B^*(x,r) \subseteq V$. So, $$p> f_V(x,0) \geqslant  f_U(x,0) = r =\lim_{n\to \infty}r_n \geqslant  \lim_{n\to \infty} f_{U_n}(x,y) > p;$$ a contradiction which finishes the proof. 
 \end{proof}

Obviously,  Proposition 
\ref{p9} gives an alternative proof that the Niemytzki plane is $\varkappa$-normal and we obtain the  following.

\begin{cor} The Niemytzki plane is $\ro$-stratifiable. \hfill $\Box$\end{cor}

Now, it seems natural to verify that the Niemytzki plane is $\varkappa$-metrizable.

\begin{thm} The Niemytzki plane  is $\varkappa$-metrizable. \end{thm}  \begin{proof} The family $\{f_V\colon V \in \RO(L)\}$ witnesses that the Niemytzki plane $L$  is $\ro$-stratifiable. We have showed that  this family satisfies conditions (1)--(3). So, it remains to show that it satisfies condition $(4)$. Fix a decreasing chain $\{U_n\colon n>0 \}$ consisting of regular open sets of the Niemytzki plane and put $W=\int \bigcap  \{U_n\colon n>0 \}$. Since we still have $W\subseteq U_n$,  we get $f_W(x)\leqslant  \inf\{f_{U_n}(x)\colon n>0\}$ for any $x\in L$. 
Fix  $x\in L$.    For each $n$, using Lemma \ref{l8},  choose $V_n \in \frak{B}$ such that $f_{U_n}(x)=f_{V_n}(x)$. If  $B((x_n,y_n), r_n) = V_n$, where $x_n \to a$ and $y_n\to b$ and $r_n\to r>0$, then we get 
$U=B((a,b),r) \subseteq W$ and $f_U(x) = \lim_{n\to +\infty} f_{V_n}(x)$. But if  $B^*(x_n,r_n) = V_n$, where $x_n \to a$ and $r_n\to r>0$, then we get 
$U=B^*(a,r)  \subseteq W$ and $f_U(x) = \lim_{n\to +\infty} f_{V_n}(x)$. Therefore $f_W(x) = \lim_{n\to +\infty} f_{U_n}(x)$. \end{proof}

\begin{pro} The Niemytzki plane is not stratifiable. \end{pro}
\begin{proof} Suppose that there exists a family of functions $$\{f_U\colon U \mbox{ is an open subset of } L\}$$  which fulfils conditions $(1), (2)$
and $(3)$.  Put $$\textstyle P_{m,n}= \{x\in \mathbb R\colon f_{B^*(x,1)}(x,y) >\frac1n \mbox{ whenever } 0\leqslant y<\frac1m \}. $$ Since $\mathbb R=\bigcup\{P_{n,m}\colon m>0 \mbox{ and } n>0\}$,
by the Baire category theorem, there exist a set $P_{n,m}$ and  an interval $(a,b)$ such that the intersection $P_{n,m} \cap (a,b)$ is dense in $(a,b)$. 
Choose $(x_k, c_k) \in B^*(a, \frac1k)$ such that $x_k \in P_{n,m} \cap (a,b)$ and $c_k<\frac1m$. Thus,  the sequence $\left((x_k,c_k)\right)$ is convergent to the point $(a,0)$ with respect to the Niemytzki plane. By condition $(2)$ we get $    f_{L\setminus\{(a,0)\}}(x_k,c_k)\geqslant f_{B^*(x_k,1)}(x_k,c_k)>\frac1{n}$;  a contradiction with  $f_{L\setminus\{(a,0)\}}(a,0)=0$. 
  \end{proof}
	
  Put
			\begin{multline*}\textstyle
  g_{B^*(a,r)}(x,y)=
  \begin{cases}
    r-\sqrt{(a-x)^2+(r-y)^2},&\mbox{if }(x,y)\in B^*(a,r)\\& \mbox{and } r\leqslant y;\\
 \left(r-\frac{r|x-a|}{\sqrt{2yr-y^2}}\right){\frac{(r-1)y+r}{r^2},}   &\mbox{if }(x,y)\in B^*(a,r)\\& \mbox{and } 0< y<r;\\
    {1},&\mbox{if }(x,y)=(a,0);\\
    0,&\mbox{for other cases}.
  \end{cases}
\end{multline*} 
    
 The family $$ \mathcal G =  \{g_ {B^*(a,r)} \colon  B^*(a,r) \in \mathcal{B}^*\}$$ is a $\mathcal{B}^*$-stratification, but it cannot be extended to an $\RO(L)$-stratification, i.e., to a family of functions which witnesses that the Niemytzki plane is $\ro$-stratifiable. Indeed, the set $V=\{(x,y)\in L\colon x>0\}$ is a regular open subset of the Niemytzki plane and $(0,0) \notin V$. Suppose that the family $\mathcal G \cup \{g_V\}$ fulfils conditions (1)--(3).  Check that $(\frac{1}{3n}, \frac{1}{6n})\in B^*(0,\frac1n) \cap 
B^*(\frac{1}{3n}, \frac{1}{3n}).$ Since $ B^*(\frac{1}{3n}, \frac{1}{3n}) \subseteq V$, we get $$\textstyle g_V(\frac{1}{3n}, \frac{1}{6n}) \geqslant g_{B^*(\frac{1}{3n}, \frac{1}{3n})}(\frac{1}{3n}, \frac{1}{6n}) >\frac12;$$ this is in conflict with continuity of $g_V$ and the equality $g_V(0,0)=0$.

\end{document}